\documentclass[a4paper,12pt]{article}

\usepackage{color}
\usepackage{makeidx}
\usepackage{bm}
\usepackage{latexsym}
\usepackage{amscd}
\usepackage{amsmath}
\usepackage{amssymb}

\usepackage{stmaryrd}

\usepackage{amsthm}
\usepackage{float}
\usepackage{graphicx}
\usepackage{psfrag}
\theoremstyle{plain}
\newtheorem{theorem}{Theorem}[section]
\newtheorem{corollary}[theorem]{Corollary}

\newtheorem{conjecture}[theorem]{Conjecture}

\theoremstyle{definition}
\newtheorem{definition}[theorem]{Definition}
\newtheorem{remark}[theorem]{Remark}
\newtheorem{example}[theorem]{Example}
\newtheorem{proposition}[theorem]{Proposition}

\newcommand{\bC}{\ensuremath{\mathbb{C}}}

\newcommand{\bP}{\ensuremath{\mathbb{P}}}

\newcommand{\bR}{\ensuremath{\mathbb{R}}}

\newcommand{\bZ}{\ensuremath{\mathbb{Z}}}

\newcommand{\scA}{\ensuremath{\mathcal{A}}}
\newcommand{\scB}{\ensuremath{\mathcal{B}}}

\newcommand{\scF}{\ensuremath{\mathcal{F}}}

\newcommand{\scL}{\ensuremath{\mathcal{L}}}

\newcommand{\sfM}{\mathsf{M}}

\newcommand{\ch}{\mathsf{ch}}

\newcommand{\Sep}{\operatorname{Sep}}
\newcommand{\Ker}{\operatorname{Ker}}

\newcommand{\A}{\scA}
\newcommand{\Coker}{\operatorname{Coker}}
\newcommand{\mult}{\operatorname{mult}}
\newcommand{\R}{\bR}
\newcommand{\RB}{\operatorname{RB}}

\definecolor{deepblue}{cmyk}{0,0.83,1,0.70}
\definecolor{gray}{cmyk}{0,0,0,0.3}
\definecolor{rred}{cmyk}{0,1,1,0}
\definecolor{chairo}{cmyk}{0,0.83,1,0.70}
\definecolor{roypur}{cmyk}{0.75,0.90,0,0.1}
\definecolor{darkorc}{cmyk}{0.40,0.80,0.20,0}
\definecolor{oliv}{cmyk}{0.64,0.00,0.75,0.56}
\definecolor{azuro}{cmyk}{1,1,0,0.46}

\title{Milnor fibers of real line arrangements}
\author{Masahiko Yoshinaga\thanks{Department of Mathematics, 
Hokkaido University, 
North 10, West 8, Kita-ku, 
Sapporo 060-0810, 
JAPAN 
E-mail: yoshinaga@math.sci.hokudai.ac.jp}}
\date{\today}
\begin{document}
\maketitle

\begin{abstract}
We study Milnor fibers of complexified real line arrangements. 
We give a new algorithm computing monodromy 
eigenspaces of the first cohomology. 
The algorithm is based on the description of 
minimal CW-complexes homotopic to the complements, 
and uses the real figure, that is, the adjacency relations 
of chambers. 
It enables us to generalize a 
vanishing result of Libgober, give new upper-bounds and 
characterize the $A_3$-arrangement in terms of 
non-triviality of Milnor monodromy. 

%\noindent
%{\bf MSC-class}: 32S22 (Primary) 52C35, 32S50 (Secondary)\\
%{\bf Keywords}: Arrangements, Chambers, Minimal CW-complex, 
%Orlik-Solomon algebras, Aomoto complex. 
\end{abstract}

\section{Introduction}

The Milnor fiber is a central object in the study of the 
topology of complex hypersurface singularities. 
In particular, the monodromy action 
on its cohomology groups has been 
intensively studied. Monodromy eigenspaces contain subtle geometric 
information. For example, for projective plane curves, 
the Betti numbers of Milnor fiber of the cone detect Zariski 
pairs \cite{eab}. In other words, Betti numbers of Milnor fiber 
of the cone of a plane curve are not in general determined by local and 
combinatorial data of singularities. 

In the theory of hyperplane arrangements, one of the central problems 
is to what extent topological invariants of the complements are 
determined combinatorially. For example, the cohomology ring 
is combinatorially determined (Orlik and Solomon \cite{orl-sol}), 
while the fundamental group is not 
(Rybnikov \cite{eab, ryb}). Between these two cases, 
local system cohomology groups and monodromy eigenspaces 
of Milnor fibers recently received a considerable amount of attention. 

There are several ways to compute monodromy eigenspaces of the 
Milnor fiber, especially for line arrangements. 
One is the topological method developed by Cohen and Suciu 
\cite{coh-suc}. They first give a presentation of the fundamental 
group of the complement. Then, using Fox calculus, they compute the 
monodromy eigenspaces. Another approach is the algebraic method, which 
computes the multiplicities of monodromy eigenvalues as the superabundance 
of singular points. 
This approach has recently been well developed, especially for 
line arrangements having only double and triple points 
\cite{lib-cat}.  

The purpose of this paper is to develop a topological method of 
computing Milnor monodromy for complexified real arrangements 
following 
Cohen and Suciu. The new ingredient is a recent study of minimal 
cell structures for the complements of complexified real 
arrangements \cite{yos-lef, yos-str}. 
By using the description of twisted minimal chain complexes, we 
obtain an algorithm which computes monodromy eigenspaces directly 
from real figures without passing through the presentations of $\pi_1$. 

The paper is organized as follows. 
In \S\ref{sec:pre} we recall a few results which 
are used in this paper. 
\S\ref{sec:main} is the main section of the paper. 
First, in \S\ref{sec:resban}, 
we introduce discrete geometric notions, the so-called 
{\em $k$-resonant band} and the {\em standing wave} on this band. 
These notions are used in \S\ref{sec:eigensp} for 
the computation of eigenspaces. 
Several consequences of our algorithm are discussed 
in \S\ref{sec:vanish}, \S\ref{subsec:upper} and 
\S\ref{sec:A3}. 
Among other things, we prove 
that if the arrangement contains more than $6$ lines and the 
cohomological monodromy action (of degree one) is non-trivial, 
then each line has at least three multiple points (see 
Corollary \ref{cor:3pts} for a precise statement). 
Such arrangements have been studied in discrete geometry as 
``configurations'', and 
several examples are provided in \cite{gru-simp, gru-conf}. 
In \S\ref{sec:exconj}, 
we apply our algorithm to arrangements appearing in papers by Gr\"unbaum 
\cite{gru-simp, gru-conf}. We also present several examples and 
conjectures.

%\cite{eab, ryb}

%Milnor fiber 

%Cohen-Suciu's examples. 
%Combinatorial pencil, 

%Combinatorial argument can say something about topological 
%structure of Milnor fiber. 

%relations with combinatorial pencils, multinets should be 
%explored. Complexification, some conjectures. 

%Give an algorithm using real structure. Directly to the eigenspaces. 
%(So far, via through fundam group is known.) 

%E. g. Combinatorial determinacy of Betti number is open. For triple 
%arrangement, recently done by Libgober. 

%Define the notion ``Multiple points''

\section{Preliminaries}
\label{sec:pre}
\subsection{Milnor fiber of arrangements}
\label{sec:milnor}

Let $\A=\{H_1, \dots, H_n\}$ be an affine line arrangement in 
$\R^2$ with the defining equation $Q_\A(x, y)=\prod_{i=1}^n\alpha_i$, 
where $\alpha_i$ is a defining linear equation for $H_i$. 
In this paper, we assume that not all lines are parallel 
(or equivalently, $\A$ has at least one intersection). 
The coning $c\A$ of $\A$ is an arrangement of $n+1$ 
planes in $\R^3$ defined by the equation 
$Q_{c\A}(x, y, z)=z^{n+1}Q(\frac{x}{z}, \frac{y}{z})$. 
The line $\{z=0\}\in c\A$ is called the line at infinity and 
is denoted by $H_\infty$. 
The space $\sfM(\A)=\bC^2\setminus\{Q_\A=0\}=
\bP_\bC^2\setminus\{Q_{c\A}=0\}$ is called the complexified 
complement. In this article, $\A$ always denotes a line arrangement 
in $\R^2$ and $c\A$ denotes a line arrangement in $\R\bP^2$. We call 
$p\in\R\bP^2$ a {\em multiple point} if the multiplicity of 
$c\A$ at $p$ (that is, the number of lines passing through $p$) 
is greater than or equal to $3$. 

\begin{definition}
$F_\A=
\{(x, y, z)\in\bC^3\mid Q_{c\A}(x, y, z)=1)\}$ is called the 
Milnor fiber of $\A$. The automorphism 
$\rho:F_\A\longrightarrow F_\A,\ (x, y, z)\longmapsto 
(\zeta x, \zeta y, \zeta z)$, with $\zeta=\exp(2\pi i/(n+1))$, 
is called the monodromy action. 
\end{definition}
The automorphism $\rho$ has order $n+1$. It generates 
the cyclic group $\langle\rho\rangle\simeq\bZ/(n+1)\bZ$. 
The monodromy $\rho$ induces a linear map 
$\rho^*:H^1(F_\A, \bC)\longrightarrow H^1(F_\A, \bC)$. 
Since $(\rho^*)^{n+1}$ is the identity, we have the 
eigenspace decomposition 
$H^1(F_\A, \bC)=\bigoplus_{\lambda^{n+1}=1}H^1(F_\A, \bC)_\lambda$, 
where $H^1(F_\A, \bC)_\lambda$ is the the set of 
$\lambda$-eigenvectors with eigenvalue $\lambda\in\bC^*$. 
When $\lambda=1$, $H^1(F_\A)_1=H^1(F_\A)^{\rho^*}$ is the 
subspace of elements fixed by $\rho^*$, which is isomorphic to 
$H^1(F_\A/\langle\rho\rangle)$. 
It is easily seen that the quotient by the monodromy 
action is 
$F_\A/\langle\rho\rangle\simeq\sfM(\A)$. 
Therefore, the $1$-eigenspace of the first cohomology is 
combinatorially determined, 
$H^1(F_\A)_1\simeq H^1(\sfM(\A))\simeq\bC^n$. 
In general, let $\scL_\lambda$ be a complex rank one local system 
associated with a representation 
$$
\pi_1(\sfM(\A))\longrightarrow\bC^*,\ 
\gamma_H\longmapsto\lambda, 
$$
where $\gamma_H$ is a meridian loop of the line $H$. 
Then it is known that 
\begin{equation}
\label{eq:eigen}
H^1(F_\A)_\lambda\simeq H^1(\sfM(\A), \scL_\lambda). 
\end{equation}
(See \cite{coh-suc} for details.) 

\subsection{Multinets and Milnor monodromy}
\label{sec:multinet}

In this section, we recall a relation between the combinatorial 
structures known as multinets and the eigenvalues of 
Milnor monodromy. We note that a $k$-multinet gives a lower 
bound on the eigenspace. 

\begin{definition}
A {\em $k$-multinet} on $c\A$ is a pair 
$(\mathcal{N}, \mathcal{X})$, where $\mathcal{N}$ is a 
partition of $c\A$ into $k\geq 3$ classes 
$\A_1, \dots, \A_k$ and $\mathcal{X}$ is a set of multiple points 
such that 
\begin{itemize}
\item[(i)] 
$|\A_1|=\cdots=|\A_k|$; 
\item[(ii)] 
$H\in\A_i$ and $H'\in\A_j$ ($i\neq j$) imply that $H\cap H'\in\mathcal{X}$; 
\item[(iii)] 
for all $p\in\mathcal{X}$, $|\{H\in\A_i\mid H\ni p\}|$ is 
constant and independent of $i$;  
\item[(iv)] 
for any $H, H'\in\A_i$ ($i=1, \dots, k$), there is a sequence 
$H=H_0, H_1, \dots, H_r=H'$ in $\A_i$ such that 
$H_{j-1}\cap H_j\notin\mathcal{X}$ for $1\leq j\leq r$. 
\end{itemize}
\end{definition}
The following is a consequence of 
\cite[Theorem 3.11]{fal-yuz} and \cite[Theorem 3.1 (i)]{dim-pap}
\begin{theorem}
\label{thm:multinet}
Suppose there exists a $k$-multinet on $c\A$ for some $k\geq 3$ and 
set $\lambda=e^{2\pi i/k}$. Then 
$$
\dim H^1(F_\A)_\lambda\geq k-2. 
$$
\end{theorem}

\subsection{Twisted minimal cochain complexes}
\label{sec:twisted}

In this section, we recall the construction of the 
twisted minimal cochain complex from 
\cite{yos-lef, yos-ch, yos-loc}, 
which will be used for the computation of the 
right hand side of (\ref{eq:eigen}). 

A connected component of $\bR^2\setminus\bigcup_{H\in\scA}H$ 
is called a chamber. The set of all chambers is denoted by 
$\ch(\scA)$. A chamber $C\in\ch(\A)$ is called bounded 
(resp. unbounded) if the area is finite (resp. infinite). 
For an unbounded chamber $U\in\ch(\A)$, the opposite unbounded 
chamber is denoted by $U^\lor$ (see \cite[Definition 2.1]{yos-loc} for 
the definition; see also Figure \ref{fig:numbering} below). 

Let $\scF$ be a generic flag in $\bR^2$
$$
\scF:
\emptyset=
\scF^{-1}\subset
\scF^{0}\subset
\scF^{1}\subset
\scF^{2}=\bR^2, 
$$
where $\scF^k$ is a generic $k$-dimensional 
affine subspace. 

\begin{definition}
For $k=0, 1, 2$, define the subset 
$\ch^k_\scF(\scA)\subset\ch(\scA)$ by 
$$
\ch^k_\scF(\scA):=
\{C\in\ch(\scA)\mid C\cap\scF^k\neq\emptyset, 
C\cap\scF^{k-1}=\emptyset\}.
$$ 
\end{definition}
The set of chambers decomposes into a 
disjoint union as 
$\ch(\scA)=
\ch^0_\scF(\scA)\sqcup
\ch^1_\scF(\scA)\sqcup
\ch^2_\scF(\scA)$. The cardinality of 
$\ch^k_\scF(\scA)$ is equal to $b_k(\sfM(\A))$ 
for $k=0, 1, 2$.

We further assume that 
the generic flag $\scF$ satisfies the following 
conditions: 
\begin{itemize}
\item $\scF^1$ does not separate intersections of $\scA$, 
\item $\scF^0$ does not separate $n$-points 
$\scA\cap\scF^1$. 
\end{itemize}
Then we can choose coordinates $x_1, x_2$ so that 
$\scF^0$ is the origin $(0,0)$, 
$\scF^1$ is given by $x_2=0$, all intersections of $\scA$ 
are contained in the upper-half plane $\{(x_1, x_2)\in\bR^2\mid
x_2>0\}$ and $\scA\cap\scF^1$ is contained in the 
half-line $\{(x_1, 0)\mid x_1>0\}$. 

We set $H_i\cap\scF^1$ to have coordinates $(a_i, 0)$. 
By changing the numbering of lines and the signs of 
the defining equation $\alpha_i$ of $H_i\in\scA$ 
we may assume that 
\begin{itemize}
\item $0<a_1<a_2<\dots<a_n$, 
\item the origin $\scF^0$ is contained in the negative 
half-plane $H_i^-=\{\alpha_i<0\}$. 
\end{itemize}
We set 
$\ch_0^\scF(\scA)=\{U_0\}$ and 
$\ch_1^\scF(\scA)=\{U_1, \dots, U_{n-1}, U_0^\lor\}$ so that 
$U_p\cap\scF^1$ is equal to the interval $(a_p, a_{p+1})$ for 
$p=1, \dots, n-1$. 
%(We use the convention $a_0=+\infty$.) 
It is easily seen that the 
chambers $U_0, U_1, \dots, U_{n-1}$ and $U_0^\lor$ 
have the following expression: 
\begin{equation}
\begin{split}
&U_0=\bigcap_{i=1}^n\{\alpha_i<0\},\\
&U_p=
\bigcap_{i=1}^{p}\{\alpha_i>0\}\cap
\bigcap_{i=p+1}^n\{\alpha_i<0\},\ (p=1, \dots, n-1),\\
&U_0^\lor=
\bigcap_{i=1}^{n}\{\alpha_i>0\}.
\end{split}
\end{equation}
The notations introduced to this point are illustrated 
in Figure \ref{fig:numbering}.

\begin{figure}[htbp]
\begin{picture}(100,100)(20,0)
\thicklines

\put(70,20){\circle*{4}}
\put(40,27){$\scF^0(0,0)$}

%F1
\multiput(50,20)(3,0){83}{\circle*{1}}
\put(300,20){\vector(1,0){0}}
\put(285,24){$\scF^1$}

%H5
\put(280,0){\line(-2,1){200}}
\put(240,20){\circle*{3}}
\put(243,24){$a_5$}
\put(278,-10){$H_5$}

%H4
\put(200,0){\line(0,1){100}}
\put(200,20){\circle*{3}}
\put(203,24){$a_4$}
\put(195,-10){$H_4$}

%H3
\put(170,0){\line(0,1){100}}
\put(170,20){\circle*{3}}
\put(173,24){$a_3$}
\put(165,-10){$H_3$}

%H2
\put(140,0){\line(0,1){100}}
\put(140,20){\circle*{3}}
\put(143,24){$a_2$}
\put(135,-10){$H_2$}

%H1
%\put(80,0){\line(1,1){100}}
\put(60,0){\line(2,1){200}}
\put(100,20){\circle*{3}}
\put(94,24){$a_1$}
\put(55,-10){$H_1$}

%C0
\put(118,52){$U_0$}
%C0^
\put(210,52){$U_0^\lor$}

%C1
\put(118,0){$U_1$}
%C2
\put(150,0){$U_2$}
%C3
\put(180,0){$U_3$}
%C4
\put(210,0){$U_4$}

%C4^
\put(118,90){$U_4^\lor$}
%C2^
\put(150,90){$U_2^\lor$}
%C3^
\put(180,90){$U_3^\lor$}
%C1^
\put(210,90){$U_1^\lor$}

%C5
\put(143,52){$C_1$}
%C6
\put(183,52){$C_2$}

%ch
\put(280,92){$\ch^0_\scF(\scA)=\{U_0\}$}
\put(280,75){$\ch^1_\scF(\scA)=\{U_0^\lor, U_1,\dots,U_4\}$}
\put(280,58){$\ch^2_\scF(\scA)=\{U_1^\lor,\dots,U_4^\lor,
C_1,C_2\}$}

\end{picture}
     \caption{Numbering of lines and chambers.}\label{fig:numbering}
\end{figure}
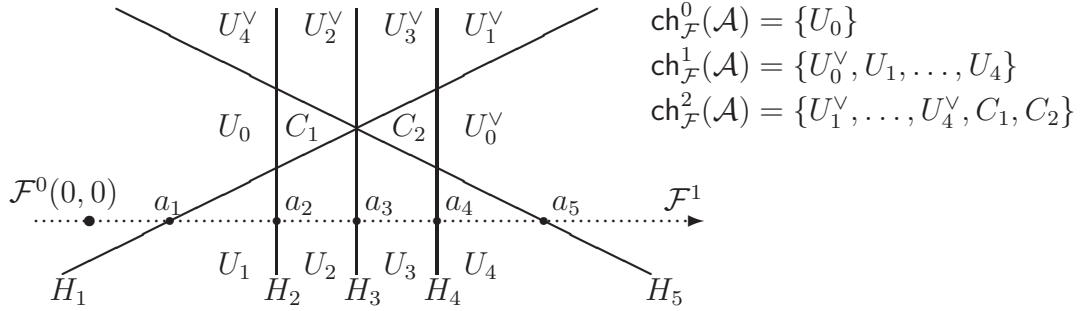

Let $\scL$ be a complex rank-one local system 
on $\sfM(\scA)$. The local system $\scL$ is determined by 
non-zero complex numbers (monodromy around $H_i$) 
$q_i\in\bC^*$, $i=1, \dots, n$. 
Fix a square root $q_i^{1/2}\in\bC^*$ for each $i$. 

\begin{definition}
\label{def:sep}
$(1)$ 
For $C, C'\in\ch(\A)$, let us denote by 
$\Sep(C,C')$ the set of lines $H_i\in\A$ which 
separate $C$ and $C'$. \\
$(2)$ 
Define the complex number $\Delta(C,C')\in\bC$ by 
$$
\Delta(C,C'):=
\prod_{H_i\in\Sep(C,C')}q_i^{1/2}
-
\prod_{H_i\in\Sep(C,C')}q_i^{-1/2}. 
$$
\end{definition}
Now we construct the cochain complex $(\bC[\ch^\bullet_\scF(\scA)], d_\scL)$. 
\begin{itemize}
\item[(i)] The map 
$d_\scL:\bC[\ch^0_\scF(\scA)]
\longrightarrow\bC[\ch^1_\scF(\scA)]$ is defined by 
$$
d_\scL([U_0])=
\Delta(U_0,U_0^\lor)[U_0^\lor]+\sum_{p=1}^{n-1}\Delta(U_0, U_p)[U_p]. 
$$
\item[(ii)] 
$d_\scL:\bC[\ch^1_\scF(\scA)]
\longrightarrow\bC[\ch^2_\scF(\scA)]$ is defined by 
\begin{equation*}
\begin{split}
d_\scL([U_p])&=
-\sum_{\substack{C\in\ch^2_\scF(\A) \\ \alpha_p(C)>0\\ \alpha_{p+1}(C)<0}}\Delta(U_p, C)[C]
+\sum_{\substack{C\in\ch^2_\scF(\A) \\ \alpha_p(C)<0\\ \alpha_{p+1}(C)>0}}\Delta(U_p, C)[C],\ (\mbox{for }p=1, \dots, n-1),
\\
d_\scL([U_0^\lor])&=
-\sum_{\alpha_{n}(C)>0}\Delta(U_0^\lor, C)[C]. 
\end{split}
\end{equation*}
\end{itemize}

\begin{example}
Let $\A=\{H_1, \dots, H_5\}$, and let the flag $\scF$ be as in Figure 
\ref{fig:numbering}. Then 
$$
d_\scL([U_0])=
([U_1], [U_2], [U_3], [U_4], [U_0^\lor])
\begin{pmatrix}
q_{1}^{1/2}-q_{1}^{-1/2}\\
q_{12}^{1/2}-q_{12}^{-1/2}\\
q_{123}^{1/2}-q_{123}^{-1/2}\\
q_{1234}^{1/2}-q_{1234}^{-1/2}\\
q_{12345}^{1/2}-q_{12345}^{-1/2}
\end{pmatrix}, 
$$
\begin{equation*}
\begin{split}
&d_\scL([U_1], [U_2], [U_3], [U_4], [U_0^\lor])
=
([U_1^\lor], [U_2^\lor], [U_3^\lor], [U_4^\lor], [C_1], [C_2])\\
&\times
\begin{pmatrix}
q_{12345}^{1/2}-q_{12345}^{-1/2}
&0
&0
&0
&-(q_{1}^{1/2}-q_{1}^{-1/2})\\ %U_1^\lor
q_{125}^{1/2}-q_{125}^{-1/2}
&-(q_{15}^{1/2}-q_{15}^{-1/2})
&0
&q_{1345}^{1/2}-q_{1345}^{-1/2}
&-(q_{134}^{1/2}-q_{134}^{-1/2})\\ %U_2^\lor
q_{1235}^{1/2}-q_{1235}^{-1/2}
&0
&-(q_{15}^{1/2}-q_{15}^{-1/2})
&q_{145}^{1/2}-q_{145}^{-1/2}
&-(q_{14}^{1/2}-q_{14}^{-1/2})\\ %U_3^\lor
0
&0
&0
&q_{12345}^{1/2}-q_{12345}^{-1/2}
&-(q_{1234}^{1/2}-q_{1234}^{-1/2})\\ %U_4^\lor
q_{12}^{1/2}-q_{12}^{-1/2}
&-(q_{1}^{1/2}-q_{1}^{-1/2})
&0
&0
&0\\ %C_1
0
&0
&-(q_{5}^{1/2}-q_{5}^{-1/2})
&q_{45}^{1/2}-q_{45}^{-1/2}
&-(q_{4}^{1/2}-q_{4}^{-1/2}) %C_2
\end{pmatrix}.
\end{split}
\end{equation*}
\end{example}

\begin{theorem}
\label{thm:twist}
Under the above notation, 
$(\bC[\ch^\bullet_\scF(\scA)], d_\scL)$ is a cochain complex and 
$$
H^k(\bC[\ch^\bullet_\scF(\scA)], d_\scL)
\simeq
H^k(M(\scA), \scL). 
$$
\end{theorem}
See \cite{yos-lef, yos-ch, yos-loc} for details.

\section{Resonant band algorithm}
\label{sec:main}

%The basic strategy is to apply Theorem \ref{thm:twist} to 
%the local system $\scL_\lambda$ for $\lambda\in\bC^*$ 
%with $\lambda^{n+1}=1$. 

Let $\A=\{H_1, \dots, H_n\}$ be an arrangement of affine lines 
in $\R^2$, and let $\scF$ be a generic flag as in \S\ref{sec:twisted}. 

Fix an integer $k>1$ with $k|(n+1)$, and set 
$\lambda=e^{2\pi i/k}$. In this section, we will give 
an algorithm for computing the $\lambda$-eigenspace 
$H^1(F_\A)_\lambda$ of the first cohomology of a Milnor fiber. 

\subsection{Resonant bands and standing waves}
\label{sec:resban}

\begin{definition}
\label{def:band}
A {\em band} $B$ is a region bounded by a pair of consecutive 
parallel lines $H_i$ and $H_{i+1}$. 
\end{definition}
Each band $B$ includes two unbounded chambers 
$U_1(B), U_2(B)\in\ch(\A)$. By definition, 
$U_1(B)$ and $U_2(B)$ are opposite each other, 
$U_1(B)^\lor=U_2(B)$ and $U_2(B)^\lor=U_1(B)$. 

Define the adjacency distance $d(C, C')$ between 
two chambers $C$ and $C'$ to be the number of 
lines $H\in\A$ that separate $C$ and $C'$, that is, 
$$
d(C, C')=|\Sep(C, C')|. 
$$
The distance $d(U_1(B), U_2(B))$ is called the length 
of the band $B$. 
\begin{remark}
\label{rem:mult}
Let $\overline{B}$ be the closure of $B$ in 
the real projective plane $\R\bP^2$. $\overline{B}$ 
intersects $H_\infty$ in one point, 
$\overline{B}\cap H_\infty$. Each line $H\in\A\cup\{H_\infty\}$ 
either passes $\overline{B}\cap H_\infty$ or 
separates $U_1(B)$ and $U_2(B)$. Therefore 
the length of $B$ is equal to 
$n+1-\mult(\overline{B}\cap H_\infty)$. 
\end{remark}

\begin{definition}
A band $B$ is called {\em $k$-resonant} if the length of $B$ 
is divisible by $k$. We denote the set of all $k$-resonant bands 
by $\RB_k(\A)$. 
\end{definition}
To a $k$-resonant band $B\in\RB_k(\A)$, we can associate 
a {\em standing wave} $\nabla(B)\in\bC[\ch(\A)]$ on the band $B$ 
as follows: 
\begin{equation}
\label{eq:wave}
\begin{split}
\nabla(B)&=
\sum_{\substack{C\in\ch(\A),\\ C\subset B}}
\left(
e^{\frac{\pi i d(U_1(B), C)}{k}}-
e^{-\frac{\pi i d(U_1(B), C)}{k}}
\right)\cdot[C]\\
&=
\sum_{\substack{C\in\ch(\A), \\C\subset B}}
\left(
\lambda^{\frac{d(U_1(B), C)}{2}}-
\lambda^{-\frac{d(U_1(B), C)}{2}}
\right)\cdot[C]\\
\\
&=
2i\cdot
\sum_{\substack{C\in\ch(\A), \\C\subset B}}
\sin\left(\frac{\pi d(U_1(B), C)}{k}\right)\cdot[C]. 
\end{split}
\end{equation}

\begin{remark}
Since the length $d(U_1(B), U_2(B))$ of the band $B$ is divisible by $k$, 
the coefficients of $[U_1(B)]$ and $[U_2(B)]$ in the linear combination 
in (\ref{eq:wave}) are zero. Hence the chambers in the summations 
in (\ref{eq:wave}) run only over bounded chambers contained in $B$. 
We also note that exchanging of $U_1(B)$ and $U_2(B)$ affects at most the 
sign of $\nabla(B)$. 
\end{remark}

\begin{remark}
To indicate the choice of $U_1(B)$ and $U_2(B)$, we always 
put the name $B$ of the band in the unbounded chamber 
$U_1(B)$ (see Figure \ref{fig:A3}). 
\end{remark}

\subsection{Eigenspaces via resonant bands}
\label{sec:eigensp}

The map $B\longmapsto\nabla(B)$ can be naturally extended to the linear map 
\begin{equation}
\label{eq:nabla}
\nabla:
\bC[\RB_k(\A)]\longrightarrow
\bC[\ch(\A)]. 
\end{equation}
%We can describe

\begin{theorem}
\label{thm:main}
The kernel of $\nabla$ is isomorphic to the $\lambda$-eigenspace 
of the Milnor fiber monodromy, that is, 
$$
\Ker\left(\nabla:
\bC[\RB_k(\A)]\longrightarrow
\bC[\ch(\A)]\right)\simeq 
H^1(F_\A)_\lambda.
$$
In particular, $\dim H^1(F_\A)_\lambda$ is equal to the number of 
linear relations among the standing waves 
$\nabla(B)$, $B\in\RB_k(\A)$. 
\end{theorem}

\begin{proof}
Let $\scL_\lambda$ be the rank-one local system on $\sfM(\A)$ 
defined by $q_1=\cdots=q_n=\lambda\in\bC^*$ 
(see \S\ref{sec:milnor} and \S\ref{sec:twisted}). In this case, 
$\Delta(C, C')$ depends only on the adjacency distance 
$d(C, C')$, or more precisely, 
$$
\Delta(C, C')=
\lambda^{\frac{d(C, C')}{2}}-
\lambda^{-\frac{d(C, C')}{2}}. 
$$
Now, we consider the first cohomology group 
$H^1(\bC[\ch^\bullet_\scF(\A)], d_\scL)$ of the twisted minimal 
cochain complex. The image $d_\scL:\bC[\ch^0_\scF(\A)]
\longrightarrow \bC[\ch^1_\scF(\A)]$ is generated by 
$$
d_\scL([U_0])=
\sum_{p=1}^{n-1}
(
\lambda^{\frac{p}{2}}-
\lambda^{-\frac{p}{2}}
)[U_p] 
+
\left(
\lambda^{\frac{n}{2}}-
\lambda^{-\frac{n}{2}}
\right)[U_0^\lor]. 
$$
Since $\lambda=e^{2\pi i/k}$ with $k>1$ and $k|(n+1)$, we have 
$\lambda^{\frac{n}{2}}-\lambda^{-\frac{n}{2}}=
\lambda^{-\frac{n}{2}}(\lambda^n-1)\neq 0$. Thus 
the coefficient of $[U_0^\lor]$ in $d_\scL([U_0])$ is non-zero. 
Define the subspace $V$ of $\bC[\ch^1_\scF(\A)]$ by 
\begin{equation}
\begin{split}
V&=\bigoplus_{p=1}^{n-1}\bC\cdot[U_p]\\
(&\simeq \Coker
\left(d_\scL:\bC[\ch^0_\scF(\A)]
\longrightarrow \bC[\ch^1_\scF(\A)]\right)). 
\end{split}
\end{equation}
Then $H^1(\bC[\ch^\bullet_\scF(\A)], d_\scL)$ is isomorphic to 
$\Ker\left(d_\scL|_V:V\longrightarrow\bC[\ch^2_\scF(\A)]\right)$. 
It is sufficient to show that 
$\Ker (d_\scL|_V)\simeq \Ker\nabla$, which will be done in several steps. 
%$\Ker\left(\nabla:\bC[\RB_k(\A)]\longrightarrow\bC[\ch(\A)]\right)$. 
Suppose that 
$\varphi=\sum_{p=1}^{n-1}c_p\cdot[U_p]\in \Ker(d_\scL|_V)$. 

\begin{itemize}
\item[(i)] 
If $H_i$ and $H_{i+1}$ are not parallel, then $c_i=0$. 
\end{itemize}
Note that if $j\neq i$, then the chamber $[U_i^\lor]$ does not appear 
in $d_\scL([U_{j}])$. Thus the coefficient of $[U_i^\lor]$ in 
$$
d_\scL(\varphi)=
\sum_{p=1}^{n-1}c_p\cdot d_\scL([U_p])
$$
is $c_i\cdot\Delta(U_i, U_i^\lor)=
c_i(\lambda^{\frac{n}{2}}-\lambda^{-\frac{n}{2}})$. This equals zero if and only if 
$c_i=0$. 

Now we may assume that 
$\varphi=\sum_{p}c_p\cdot[U_p]\in\Ker(d_\scL)$ is a linear combination 
of $[U_p]$s such that $H_p$ and $H_{p+1}$ are parallel. 
Suppose that $H_i$ and $H_{i+1}$ are parallel and denote by  
$B_i$ the band determined by these lines. 
\begin{itemize}
\item[(ii)] 
If $B_i$ is not $k$-resonant, then $c_i=0$. 
\end{itemize}
In this case, $\Delta(U_i, U_i^\lor)=
\lambda^{\frac{d(U_i, U_i^\lor)}{2}}-
\lambda^{-\frac{d(U_i, U_i^\lor)}{2}}$. By the assumption that 
$d(U_i, U_i^\lor)$ is not divisible by $k$, we have 
$\Delta(U_i, U_i^\lor)\neq 0$. Since 
$\varphi$ is a linear combination of $[U_p]$s with 
parallel boundaries $H_p$ and $H_{p+1}$, the term $[U_i^\lor]$ 
appears only in $d_\scL([U_i])$, which is equal to 
$c_i\cdot\Delta(U_i, U_i^\lor)[U_i^\lor]$. Therefore $c_i=0$. 

Finally we may assume that $\varphi$ is a linear combination of 
$[U_p]$s such that the boundaries $H_p$ and $H_{p+1}$ are parallel and 
the length of the corresponding band $B_p$ is divisible by $k$. In this case, 
it is straightforward to check that the maps $d_\scL$ and $\nabla$ are 
identical. This completes the proof. 
\end{proof}

\begin{example}
\label{ex:A(6,1)}
($A_3$-arrangement, $\A(6,1)$ or $\scB_6$) 
The three arrangements in Figure \ref{fig:A3} are projectively 
equivalent, and are respectively called $A_3$-arrangement, 
$\A(6,1)$ or $\scB_6$. (See \S \ref{sec:exconj} for the latter 
two notations.) 
We use the left figure to compute $\dim H^1(F_\A)_\lambda$. (The 
symbol $\infty$ indicates that the line at infinity is an element 
of $\A$.) Since $|c\A|=n+1=6$, $k\in\{2,3,6\}$ and we have 
$\RB_2(\A)=\RB_6(\A)=
\emptyset$, $\RB_3(\A)=\{B_1, B_2\}$. By definition, we have 
\begin{equation*}
\begin{split}
\nabla(B_1)&=\sqrt{-3}\cdot [C_1]+\sqrt{-3}\cdot [C_2]\\
\nabla(B_2)&=\sqrt{-3}\cdot [C_1]+\sqrt{-3}\cdot [C_2]. 
\end{split}
\end{equation*}
Hence we have a linear relation $\nabla(B_1-B_2)=0$ and 
$\dim H^1(F_\A)_\lambda=1$ for $\lambda=e^{2\pi i/3}$. 
(Hence the $A_3$-arrangement is pure-tone; see Definition 
\ref{def:pure-tone}.) 
\begin{figure}[htbp]
\begin{picture}(400,150)(20,0)
\thicklines

\multiput(55,0)(40,0){2}{\line(0,1){150}}
\multiput(0,55)(00,40){2}{\line(1,0){150}}
\put(0,0){\line(1,1){150}}
\put(0,140){\huge $\infty$}

\put(70,140){\large $B_1$}
\put(58,105){\color{blue}$U_1(B_1)$\normalcolor}
\put(58,35){\color{blue}$U_2(B_1)$\normalcolor}

\put(-10,70){\large $B_2$}
\put(15,70){\color{blue}$U_1(B_2)$\normalcolor}
\put(110,70){\color{blue}$U_2(B_2)$\normalcolor}

\put(60,80){\color{blue}$C_1$\normalcolor}
\put(80,65){\color{blue}$C_2$\normalcolor}

\put(170,40){\line(1,0){150}}
\qbezier(180,22.68)(180,22.68)(253.51,150)
\qbezier(310,22.68)(310,22.68)(236.49,150)
\put(245,150){\line(0,-1){130}}
\qbezier(175,31.34)(175,31.34)(300,103.51)
\qbezier(315,31.34)(315,31.34)(190,103.51)

\multiput(330,70)(0,40){2}{\line(1,0){80}}
\multiput(350,50)(40,0){2}{\line(0,1){80}}
\put(330,50){\line(1,1){80}}
\put(330,130){\line(1,-1){80}}

\end{picture}
      \caption{The $A_3$-arrangement ($=\A(6,1)=\mathcal{B}_6$)}
\label{fig:A3}
\end{figure}
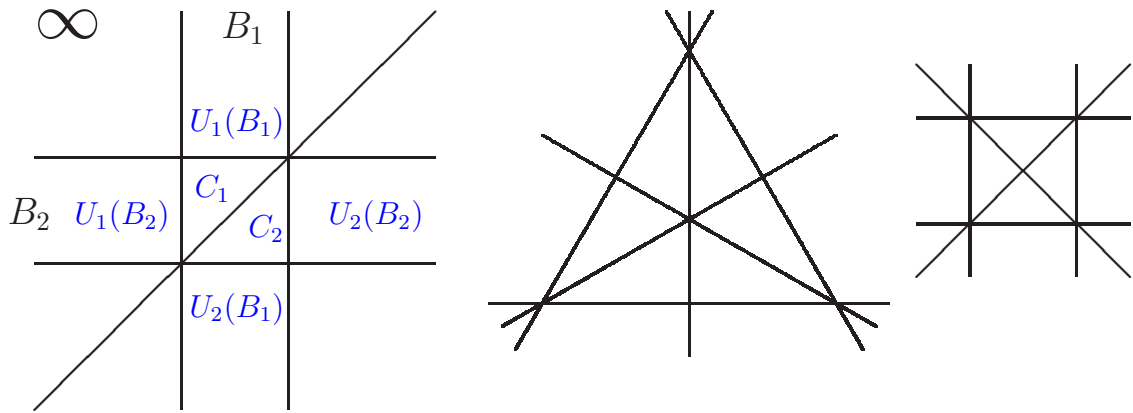
\end{example}

\begin{example}
\label{ex:A(12,2)}
($\A(12,2)$ from \cite{gru-simp}) 
Let $\A$ be the line arrangement in Figure 
\ref{fig:A(12,2)} (together with the line at infinity). 
Then $|c\A|=n+1=12$. There are seven bands, $B_1, \dots, B_7$. 
Among them, $B_5, B_6$ and $B_7$ have length $7$ which is 
coprime with $12$ so we can ignore them. We have 
$\RB_3(\A)=\{B_1, B_4\}$ and 
$\RB_2(\A)=\RB_4(\A)=\{B_2, B_3\}$. First consider the 
case $k=3$. Then 
\begin{equation*}
\begin{split}
\nabla(B_1)&=\sqrt{-3}\cdot [C_1]+\dots, \\
\nabla(B_4)&=\sqrt{-3}\cdot [C_6]+\dots. 
\end{split}
\end{equation*}
Since the chamber $C_6$ is not contained in the band $B_1$, 
it does not appear in the linear combination for $\nabla(B_1)$. 
Hence $\nabla(B_1)$ and $\nabla(B_4)$ are linearly 
independent. We conclude that 
$H^1(F_\A)_\lambda=0$ for $\lambda=
e^{2\pi i/3}$. The cases $k=2$ and $k=4$ are similar. More precisely, 
since $B_2, B_3\in\RB_2(\A)=\RB_4(\A)$ are parallel and they do not overlap, 
$\nabla(B_2)$ and $\nabla(B_3)$ are linearly independent. 
Consequently 
we have $H^1(F_\A)_{\neq 1}=0$ and so the cohomology does not 
have non-trivial eigenvalues. 
\begin{figure}[htbp]
\begin{picture}(400,150)(20,0)
\thicklines

\multiput(140,0)(40,0){4}{\line(0,1){150}}
\put(135,-8){\footnotesize $H_1$}
\multiput(100,45)(0,30){3}{\line(1,0){200}}
\put(100,15){\line(4,3){180}}
\put(120,0){\line(4,3){180}}
\put(100,135){\line(4,-3){180}}
\put(120,150){\line(4,-3){180}}
\put(295,115){\huge $\infty$}

\put(111,133){$B_1$}
\put(111,85){$B_2$}
\put(111,55){$B_3$}
\put(111,7){$B_4$}
\put(155,0){$B_5$}
\put(195,0){$B_6$}
\put(235,0){$B_7$}

\put(148,112){\small $C_1$}
\put(164,92){\small $C_2$}
\put(148,82){\small $C_3$}
\put(148,63){\small $C_4$}
\put(164,50){\small $C_5$}
\put(148,33){\small $C_6$}

\end{picture}
      \caption{$\A(12,2)$}
\label{fig:A(12,2)}
\end{figure}
\end{example}

%\begin{example}
%\label{ex:A(18,1)}
%($\A(18,1)$) Skip? 
%\end{example}

%\begin{example}
%\label{ex:A(24,1)}
%($\A(24,1)$) Skip? 
%\end{example}

The argument used in Example \ref{ex:A(12,2)} is 
generalized in the next 
section. See \S\ref{sec:exconj} for 
further examples.

\begin{remark}
\label{rem:change}
The cohomology of the Milnor fiber 
$H^1(F_\A)$ depends only on the projective arrangement 
$\A\cup\{H_\infty\}$. The change of the line at infinity 
$H_\infty$ sometimes makes the structure of resonant bands 
$\RB_k$ simpler. This fact will be used in Corollary 
\ref{cor:config}. 
\end{remark}

\subsection{Vanishing}
\label{sec:vanish}

Fix $k$ and $\lambda$ as above. 
%The following are 
We describe some corollaries to Theorem \ref{thm:main}. 

\begin{corollary}
\label{cor:empty}
If $\RB_k(\A)=\emptyset$, then $H^1(F_\A)_\lambda=0$. 
\end{corollary}
\begin{proof}
Since $\bC[\RB_k(\A)]=0$, obviously 
$\Ker(\nabla:\bC[\RB_k(\A)]\rightarrow\bC[\ch(\A)])=0$. 
By Theorem \ref{thm:main}, $H^1(F_\A)_\lambda=0$. 
\end{proof}
Using the interpretations in Remark \ref{rem:mult}, we have 
the following. 
\begin{proposition}
\label{prop:multi}
A band $B$ is $k$-resonant if and only if $\mult(\overline{B}\cap H_\infty)$ 
is divisible by $k$. 
\end{proposition}

\begin{corollary}
\label{cor:lib}
Suppose that there are no points on $H_\infty$ where the 
multiplicity of 
$c\A=\A\cup\{H_\infty\}$ is divisible by $k$. Then 
$H^1(F_\A)_\lambda=0$. 
\end{corollary}
\begin{proof}
By Proposition \ref{prop:multi}, the assumption is equivalent 
to $\RB_k(\A)=\emptyset$. 
%This is straightforward from Proposition \ref{prop:multi} and 
We then use Corollary \ref{cor:empty}. 
\end{proof}

\begin{remark}
Corollary \ref{cor:lib} 
is proved by Libgober \cite[Corollary 3.5]{lib-mil} for 
more general complex arrangement cases. 
\end{remark}

For the real case, we obtain a stronger result as follows. 

\begin{theorem}
\label{thm:1pt}
Suppose that all $k$-resonant bands are parallel to each other. 
Then $H^1(F_\A)_\lambda=0$. 
\end{theorem}
\begin{proof}
By the assumption, 
$\RB_k(\A)=\{B_1, \dots, B_m\}$ consists of parallel bands. 
Now, the supports of $\nabla(B_1), \dots, \nabla(B_m)$, that is, 
the set of chambers appearing in each standing wave, 
are mutually disjoint. They 
are obviously linearly independent. 
(Recall that, in this paper, 
we assume that the arrangement $\A$ has at least one 
intersection.) Hence $H^1(F_\A)_\lambda=0$. 
(See Example \ref{ex:A(12,2)}.) 
\end{proof}

\begin{corollary}
\label{cor:1pt}
Suppose that there is at most one point $p\in H_\infty$ 
such that the multiplicity of $c\A=\A\cup\{H_\infty\}$ at $p$ 
is divisible by $k$. Then $H^1(F_\A)_\lambda=0$. 
\end{corollary}
\begin{proof}
Again by Proposition \ref{prop:multi}, the assumption is equivalent 
to $\RB_k(\A)$ consisting of parallel bands. 
We then use Theorem \ref{thm:1pt}. 
\end{proof}
Let us denote by $H^1(F_\A)_{\neq 1}=
\bigoplus\limits_{\lambda\neq 1}H^1(F_\A)_\lambda$ the direct 
sum of non-trivial eigenspaces. The following is immediate from 
Corollary \ref{cor:1pt} and Remark \ref{rem:change}. 
\begin{corollary}
\label{cor:config}
\begin{itemize}
\item[$(1)$] 
Suppose that $H^1(F_\A)_\lambda\neq 0$. Then each line 
$H\in c\A=\A\cup\{H_\infty\}$ has at least two multiple points, such that 
the multiplicity is divisible by $k$.

\item[$(2)$] 
Suppose that $H^1(F_\A)_{\neq 1}\neq 0$. Then each line 
$H\in c\A=\A\cup\{H_\infty\}$ has at least two multiple points. %such that 
%the multiplicity is greater or equal to $3$. 
\end{itemize}
\end{corollary}
\begin{proof}
Let $H$ be such a line. Choose an affine open set 
in such a way that $H$ is a line at infinity. 
\end{proof}

\begin{remark}
We do not know whether Corollary \ref{cor:config} holds for 
complex arrangements. We will prove a stronger result in 
\S\ref{sec:A3}. 
\end{remark}

\subsection{Upper-bound}
\label{subsec:upper}

Recall that two lines $H, H'$ in the real projective plane $\R\bP^2$ 
divide the space into two regions. 

\begin{definition}
\label{def:sharp}
Let $c\A$ be a line arrangement in the real projective plane $\R\bP^2$. 
Then the pair of lines $H_i, H_j\in c\A$ is said to be a 
{\em sharp pair} if all intersection points of $c\A\setminus 
\{H_i, H_j\}$ are contained in one of two regions or lie 
on $H_i\cup H_j$. (In other words, there are no intersection points 
in one of the two regions determined by $H_i$ and $H_j$.) 
%of $\A\setminus \{H_i, H_j\}$ in the interior of the other region.) 
\end{definition}

\begin{example}
A fiber-typer arrangement has a sharp pair of lines. 
\end{example}

\begin{example}
In the Pappus arrangement (Figure \ref{fig:pappus}), 
the line at infinity and the leftmost vertical line form a 
sharp pair. So do the two boundary lines of the band $B_1$. 
Furthermore, all line arrangements appearing in this 
paper contain sharp pairs of lines. (There also exist arrangements 
which have no sharp pairs.) 
\end{example}

\begin{theorem}
\label{thm:sharp}
Assume that the arrangement $c\A$ contains a sharp pair of lines. 
Then: 
\begin{itemize}
\item[(i)] 
$\dim H^1(F)_\lambda\leq 1$ for $\lambda\neq 1$. 
\item[(ii)] 
Suppose that the pair $H_1, H_2\in c\A$ is sharp. Let $p=H_1\cap H_2$ 
be the intersection. If the multiplicity of $c\A$ at $p$ is not 
divisible by $k$, then $H^1(F)_\lambda=0$ for 
$\lambda=e^{2\pi i/k}$. 
\end{itemize}
\end{theorem}
\begin{proof}
By the $PGL_3(\bC)$ action, we may assume that 
the line at infinity $H_\infty$ and $H_1=\{x=0\}$ 
form a sharp pair and that there are no intersections in 
the region $\{(x,y)\in\R^2\mid x<0\}$ (see Figure \ref{fig:A(12,2)}). 
The intersection is $p=H_\infty\cap H_1=\{(0:1:0)\}$. Let $B$ be 
a horizontal (that is, non-vertical) band, 
that does not passing through 
the point $p$. Denote by $C_B$ the leftmost bounded chamber 
in $B$ (e.g., in Figure \ref{fig:A(12,2)}, $C_{B_1}=C_1, 
C_{B_2}=C_3, C_{B_3}=C_4$ and $C_{B_4}=C_6$). 

First, consider the case where the multiplicity of $c\A$ at $p$ is 
not divisible by $k$. Then all $k$-resonant bands are horizontal. 
Let $B\in\RB_k(\A)$. Then 
\begin{equation}
\label{eq:nozero}
\nabla(B)=
2i\sin
\left(\frac{\pi}{k}\right)\cdot[C_B]+\cdots, 
\end{equation}
and so $[C_B]$ has a non-zero coefficient. 
Since $C_B$ is contained in the unique $k$-resonant band $B$, 
$[C_B]$ does not appear in the linear combinations of 
other $k$-resonant bands. Hence 
$\nabla(B), B\in\RB_k(\A)$ are linearly independent. 
Thus (ii) is proved. 

Now we assume that the multiplicity of $\A$ at $p$ is 
divisible by $k$. In this case, there are vertical 
$k$-resonant bands. Denote by $B_{left}$ the 
leftmost vertical band 
(in Figure \ref{fig:A(12,2)}, $B_{left}=B_5$). Suppose that 
$$
c_{left}\cdot B_{left}+\cdots \in\Ker(\nabla). 
$$
Let $B\in\RB_k(\A)$ be a horizontal $k$-resonant band. 
Then, since $C_B$ is contained in only $B$ and $B_{left}$, 
the coefficient $c_{left}$ of $B_{left}$ determines 
the coefficient of $B$. The coefficients of other 
vertical $k$-resonant bands are also determined by 
those of the horizontal bands. Hence $\Ker(\nabla)$ is 
at most one-dimensional. 
\end{proof}

\begin{example}
\label{ex:discrim}
Let $\A$ be as in Figure \ref{fig:discrim}, with $|c\A|=12$. 
Let $k=3$. Then $\RB_3=\{
B^1_1, B^1_2, B^1_3, B^1_4, 
B^2_1, B^2_2, B^2_3, B^2_4\}$ contains eight bands. Suppose that 
$\sum_{i=1}^2\sum_{j=1}^4 c_{ij}[B^i_j]\in\Ker(\nabla)$. 
By computing 
$\sum_{i=1}^2\sum_{j=1}^4 c_{ij}\nabla(B^i_j)$ 
as in the figure, we conclude that 
all the coefficients are $c_{ij}=0$. Hence 
$H^1(F_\A)_\lambda=0$ for $\lambda=e^{2\pi i/3}$. 
Note that the multiple points on the diagonal line are 
triple points. 
If we put the diagonal line at infinity, then $\RB_2=\RB_4=\RB_6=\emptyset$. 
%if we
%(and a double point on $H_\infty$), 
Therefore $H^1(F_\A)_{-1}=H^1(F_\A)_i=H^1(F_\A)_{e^{2\pi i/6}}=0$ 
by Corollary \ref{cor:config}. 

\begin{figure}[htbp]
\begin{picture}(400,200)(0,0)
\thicklines

\multiput(50,40)(0,30){5}{\line(1,0){300}}
\multiput(100,0)(50,0){5}{\line(0,1){200}}
\put(50,10){\line(5,3){300}}
\put(310,185){\huge $\infty$}

\color{blue}
\put(120,5){$B^1_1$}
\put(170,5){$B^1_2$}
\put(220,5){$B^1_3$}
\put(270,5){$B^1_4$}

\color{red}
\put(60,50){$B^2_1$}
\put(60,80){$B^2_2$}
\put(60,110){$B^2_3$}
\put(60,140){$B^2_4$}

\normalcolor
\color{blue}
\put(115,43){\scriptsize ${\color{blue}c_{11}}{\color{red}+c_{21}}$}
\put(102,63){\scriptsize ${\color{blue}c_{11}}{\color{red}+c_{21}}$}
\put(118,83){\scriptsize ${\color{red} c_{22}}$}
\put(105,113){\scriptsize ${\color{blue}-c_{11}}{\color{red}+c_{23}}$}
\put(105,143){\scriptsize ${\color{blue}-c_{11}}{\color{red}+c_{24}}$}

\put(175,73){\scriptsize ${\color{blue}c_{12}}$}
\put(168,53){\scriptsize ${\color{blue} c_{12}}$}
\put(162,93){\scriptsize ${\color{red}c_{22}}$}
\put(155,113){\scriptsize ${\color{blue}-c_{12}}{\color{red}+c_{23}}$}
\put(155,143){\scriptsize ${\color{blue}-c_{12}}{\color{red}+c_{24}}$}

\put(210,53){\scriptsize ${\color{blue}c_{13}}{\color{red}-c_{21}}$}
\put(210,83){\scriptsize ${\color{blue}c_{13}}{\color{red}-c_{22}}$}
\put(225,103){\scriptsize ${\color{red}-c_{23}}$}
\put(212,123){\scriptsize ${\color{blue}-c_{13}}$}
\put(218,143){\scriptsize ${\color{blue} -c_{13}}$}

\put(260,53){\scriptsize ${\color{blue}c_{14}}{\color{red}-c_{21}}$}
\put(260,83){\scriptsize ${\color{blue}c_{14}}{\color{red}-c_{22}}$}
\put(268,113){\scriptsize ${\color{red}-c_{23}}$}
\put(260,132){\scriptsize ${\color{blue}-c_{14}}{\color{red}-c_{24}}$}
\put(250,154.5){\scriptsize ${\color{blue}-c_{14}}{\color{red}-c_{24}}$}

\end{picture}
      \caption{Example \ref{ex:discrim}}
\label{fig:discrim}
\end{figure}
\end{example}

\subsection{A characterization of the $A_3$-arrangement}
\label{sec:A3}

Now we give a characterization of the $A_3$-arrangement in terms of non-trivial 
Milnor monodromy. 

\begin{theorem}
\label{thm:A3}
Assume that $H^1(F_\A)_\lambda\neq 0$ with 
$\lambda=e^{2\pi i/k}\neq 1$, and that 
the set of $k$-resonant bands $\RB_k(\A)$ consists of 
at most two directions (this condition is equivalent 
to $H_\infty$ containing at most two multiple points 
which have multiplicities divisible by $k$). 
Then $c\A$ is equivalent to the $A_3$-arrangement. 
\end{theorem}

\begin{proof}
If $\RB_k(\A)$ consists of one direction, then by Theorem \ref{thm:1pt}, 
$H^1(F_\A)_\lambda=0$. Thus we may assume that $\RB_k(\A)$ 
consists of two directions. After a suitable change of coordinates, 
we assume the following (see Figure \ref{fig:proof}): 
\begin{itemize}
\item 
$\RB_k(\A)=\{B^1_1, B^1_2, \dots, B^1_p, B^2_1, B^2_2, \dots, B^2_q\}$. 
\item 
$B^1_1, B^1_2, \dots, B^1_p$ are parallel to the 
vertical line $x=0$ and may be expressed as 
$B^1_i=\{(x, y)\in\R^2\mid a_i<x<a_{i+1}\}$ with 
$a_1<\dots<a_{p+1}$. The lines 
$H^1_i=\{x=a_i\}$, $i=1, \dots, p+1$, which are vertical lines, are boundaries of 
these bands. 
\item 
$B^2_1, B^2_2, \dots, B^2_q$ are parallel to the 
horizontal line $y=0$ and may be expressed as 
$B^2_i=\{(x, y)\in\R^2\mid b_i<y<b_{i+1}\}$ with 
$b_1<\dots<b_{q+1}$. The lines 
$H^2_i=\{y=b_i\}$, $i=1, \dots, q+1$, which are horizontal lines, are boundaries of 
these bands. 
\item 
Let 
$\sum\limits_{i=1}^p c_{1i}\cdot B^2_i+
\sum\limits_{i=1}^q c_{2i}\cdot B^2_i\in\Ker(\nabla)$ be a 
non-trivial relation among $k$-resonant bands. 
\end{itemize}

\begin{figure}[htbp]
\begin{picture}(400,150)(0,0)
\thicklines

\multiput(70,20)(30,0){5}{\line(0,1){130}}
\multiput(10,75)(0,25){3}{\line(1,0){230}}

\color{red}
\put(10,50){\line(1,0){60}}
\put(38,52){\small $\sigma$}

\normalcolor
\put(70,50){\line(1,0){170}}

\put(80,25){$B^1_1$}
\put(110,25){$B^1_2$}
\put(140,25){$\cdots$}
\put(170,25){$B^1_p$}

\put(65,8){\small $H^1_1$}
\put(95,8){\small $H^1_2$}
%\put(125,8){\small $H^1_3$}
%\put(155,8){\small $H^1_p$}
\put(185,8){\small $H^1_{p+1}$}

\put(15,58){$B^2_1$}
\put(15,83){$\ \ \vdots$}
\put(15,108){$B^2_q$}

\put(-3,47){\small $H^2_1$}
\put(-3,122){\small $H^2_{q+1}$}

\put(20,20){$K$}
\put(30,30){\line(1,1){120}}
\put(60,52){\small $C$}

%right
\multiput(280,20)(50,0){3}{\line(0,1){130}}
\color{blue}
\put(305,20){\line(0,1){130}}
\normalcolor
\put(355,20){\line(0,1){130}}
\multiput(260,50)(0,40){3}{\line(1,0){140}}
\put(260,34){\line(5,4){140}}
\put(253,25){$K$}

\end{picture}
      \caption{Proof of Theorem \ref{thm:A3}}
\label{fig:proof}
\end{figure}
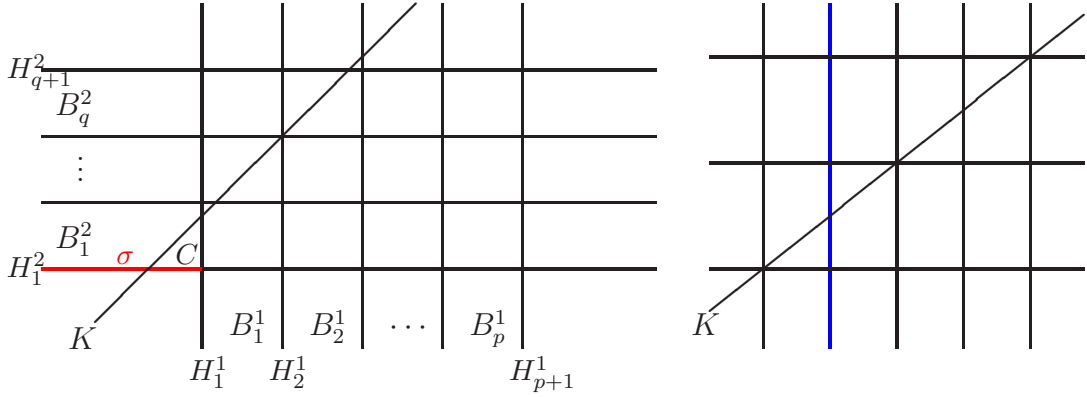

The multiplicity of $c\A$ at 
$(0:1:0)=\overline{B^1_i}\cap H_\infty$ is 
$p+2$, which must be divisible by $k$ by Proposition \ref{prop:multi}. 
Hence $p$ can be expressed as $p=ks-2$ $(s\in\bZ_{>0})$. 
Similarly, set $q=kt-2$ $(t\in\bZ_{>0})$. So far, together with $H_\infty$, 
we have $(s+t)k-1$ lines. The remaining $n+1-(s+t)k+1$ can be expressed 
as $ku+1$, which in particular cannot be zero. 
We prove that (1) $u=0$, (2) $k=3$, (3) $p=q$, (4) $p=q=1$, 
and conclude that $c\A$ is the $A_3$-arrangement. 

\medskip

(1) We first prove that $u=0$. Consider the open segment 
$$
\sigma=\{(x:y:1)\in\R^2\mid y=b_1, -\infty<x<a_1\}
\subset H^2_1 
$$
which is bounded by the two points 
$(1:0:0)$ and $(a_1:b_1:1)$. (See Figure \ref{fig:proof}.) 
Let us prove that there 
are no intersections on $\sigma$. 
Suppose that the line $K\in\A$ 
intersects $\sigma$. The leftmost chamber $C$ in $B^2_1$ is not contained 
in the other $k$-resonant bands and satisfies $d(U_1(B^2_1),C)=1$. 
Since 
$$
\nabla(B^1_1)=2i\sin\left(
\frac{\pi}{k}
\right)\cdot[C]+\cdots, 
$$
and the coefficient of $[C]$ is non-zero, 
we have $c_{21}=0$. This implies that 
$$
c_{11}=c_{12}=\cdots=c_{1p}=0. 
$$
Then we have $c_{22}=\cdots=c_{2q}=0$. This contradicts the hypothesis that 
$H^1(F_\A)_\lambda\neq 0$. This contradiction proves that 
there are no intersections on the segment $\sigma$. Similarly 
there are no intersections on the seven other similar segments, that is, 
the boundaries of the four regions 
\begin{equation*}
\begin{split}
\{(x:y:1)\mid x<a_1, y<b_1\},\ &
\{(x:y:1)\mid x<a_1, y>b_q\},\\
\{(x:y:1)\mid x>a_p, y<b_1\},\ &
\{(x:y:1)\mid x>a_p, y>b_q\}. 
\end{split}
\end{equation*}
%e.g. 
%$\{(x:y:1)\in\R^2\mid x=a_1, b_q<x<+\infty\}$, etc. 
Thus $K\in\A$ must be one of the two diagonals 
\begin{equation*}
\begin{split}
K_1&=\mbox{the line connecting $(a_1:b_1:1)$ and $(a_p:b_q:1)$}, 
\\
K_2&=\mbox{the line connecting $(a_p:b_1:1)$ and $(a_1:b_q:1)$}. 
\end{split}
\end{equation*}
%connecting $(a_1:b_1:1)$ and $(a_p:b_q:1)$ and 
%connecting $(a_p:b_1:1)$ and $(a_1:b_q:1)$. % (or both of them). 
Hence $ku+1\leq 2$, and we have $u=0$. 

(2) Now we prove $k=3$. Using the above notation, we may 
assume that $\A=\{H^1_1, \dots, H^1_{ks-1}, 
H^2_1, \dots, H^2_{kt-1}, K\}$, where $K$ is the diagonal line 
connecting $(a_1:b_1:1)$ and $(a_p:b_q:1)$. Then the point 
$(a_1:b_1:1)$ has multiplicity $3$. The line $H^1_1$ has 
exactly two multiple points, $(a_1:b_1:1)$ and $(0:1:0)$. 
By Corollary \ref{cor:config}, $k$ is a common divisor of 
$3$ and the multiplicity of $(0:1:0)$. Since $k\neq 1$, we have $k=3$. 

(3) If $p\neq q$, then there exists a (either vertical or horizontal) 
line which intersects the diagonal line $K$ normally (that is, 
with multiplicity $2$, the right-hand side of Figure \ref{fig:proof}). 
Then the line has only one multiple point on $H_\infty$ (either 
$(0:1:0)$ or $(1:0:0)$). 
This contradicts Corollary \ref{cor:config}. Hence 
$p=q$. 

(4) If $p=q>1$, then we can prove that $H^1(F_\A)_{e^{2\pi i/3}}=0$ 
by an argument similar to Example \ref{ex:discrim}. Hence 
$p=q=1$. This obviously implies that $c\A$ is isomorphic to 
the $A_3$-arrangement. 
\end{proof}

\begin{corollary}
\label{cor:3pts}
Assume that $\A$ is a real arrangement as above, 
and assume that $|c\A|=n+1\geq 7$. If $H^1(F_\A)_{\lambda}\neq 0$, then 
each line $H\in c\A$ passes through at least three multiple points 
which have multiplicities divisible by $k$. 
\end{corollary}

\begin{remark}
We do not know whether Theorem \ref{thm:A3} and 
Corollary \ref{cor:3pts} hold for 
complex arrangements. 
\end{remark}

\section{Examples and Conjectures}
\label{sec:exconj}
%\subsection{All known real examples with nontrivial monodromy}

By the previous result (Corollary \ref{cor:3pts}), the Milnor 
fiber cohomology has non-trivial eigenspaces 
only when each line has at least 
three multiple points. 
Classes of line arrangements 
known as ``simplicial arrangements'' and ``configurations'' 
provide such examples. In this section, we present examples of 
non-trivial eigenspaces $H^1(F_\A)_{\neq1}\neq0$.

\subsection{Observation}

As far as the author knows, all examples of real arrangements 
with $H^1(F_\A)_{\neq 1}\neq 0$ have the following 
``pure-tone'' property, that is, only the third root of $1$ appears with 
multiplicity one.  
%Following examples are unifiedly treated by the next results. 
%Unified way structure. 
%As far as the author knows, all nontrivial example 
%All examples satisfy the following pure-tone.  

\begin{definition}
\label{def:pure-tone}
$\A$ is said to be {\em pure-tone} if 
$H^1(F_\A)_\lambda=0$ for $\lambda^3\neq 1$ and 
$\dim H^1(F_\A)_\lambda=1$ for $\lambda=e^{\pm 2\pi i/3}$. 
\end{definition}
Furthermore, it is observed that 
all known examples with $H^1(F_\A)_{\neq 1}\neq 0$ 
satisfy 
\begin{itemize}
\item $c\A$ has a sharp pair of lines, 
\item $c\A$ has a $k$-multinet structure with $k=3$. 
\end{itemize}
These two properties imply by Theorem \ref{thm:sharp} and Theorem 
\ref{thm:multinet} that $\A$ is pure-tone. 

\subsection{Simplicial arrangements}
\label{sec:simp}

Let $\A=\{H_1, \dots, H_n\}$ be a line arrangement in $\R^2$. 
Then the projective arrangement $c\A=\A\cup\{H_\infty\}$ in the 
real projective plane $\R\bP^2$ 
is called {\em simplicial} if each chamber is a triangle. 
Gr\"unbaum \cite{gru-simp} presents a catalogue of 
known simplicial arrangements with up to $37$ lines 
(see \cite{cun} for additional information). 

\medskip

\noindent
{\bf Notation.} The symbol $\infty$ in a figure indicates that 
the $(n+1)$-st line is $H_\infty$. The notation $\A(n,k)$ 
comes from \cite{gru-simp}, which is the $k$-th simplicial 
arrangement of $n$-lines. 

\medskip

Example \ref{ex:A(6,1)} can be generalized in two ways. 

\begin{definition}
\label{def:A(2n,1)}
For a positive integer $n\in\bZ_{>0}$, 
$\A(2n,1)$ is described as follows. Starting with 
a regular convex $n$-gon in the Euclidean plane, 
$\A(2n,1)$ is obtained by taking $n$ lines 
determined by the sides of the $n$-gon together with the 
$n$-lines of symmetry of that $n$-gon. 
$\A(2n,1)$ is a simplicial arrangement of $2n$-lines. 
\end{definition}
Obviously, the $A_3$-arrangement is equivalent to $\A(6,1)$. 

\begin{example}
\label{ex:A(12,1)}
Let $c\A=\A(12,1)$ (Figure \ref{fig:A(12,1)}). 
Then $\RB_3(\A)=\{B_1, \dots, B_7\}$. 
\begin{equation*}
\begin{array}{clccccccccr}
\nabla(B_2)=&\sqrt{-3}(C_1&&+C_3&&-C_5&-C_6&&+C_8&&+C_{10})\\
\nabla(B_3)=&\sqrt{-3}(C_1&&+C_3&&&&-C_7&&-C_9&)\\
\nabla(B_6)=&\sqrt{-3}(&C_2&&+C_4&&&&-C_8&&-C_{10})\\
\nabla(B_7)=&\sqrt{-3}(&C_2&&+C_4&-C_5&-C_6&+C_7&&+C_9&)
\end{array}
\end{equation*}
Hence we have a linear relation 
$$
\nabla(B_2)-\nabla(B_3)
+\nabla(B_6)-\nabla(B_7)=0, 
$$
and so we have that $\A(12,1)$ is pure-tone. 
\begin{figure}[htbp]
\begin{picture}(400,200)(0,0)
\thicklines

\color{oliv}
\put(-10,75){\line(1,0){200}}
\put(-10,125){\line(1,0){200}}
\put(90,0){\line(0,1){200}}
%\multiput(46.69746,0)(28.86836,0){4}{\line(0,1){200}}

\color{blue}
\put(46.69746,0){\line(0,1){200}}
\color{red}
\put(75.56582,0){\line(0,1){200}}
\color{blue}
\put(104.43418,0){\line(0,1){200}}
\color{red}
\put(133.30254,0){\line(0,1){200}}

\color{red}
\qbezier(190,199.1)(190,199.1)(-10,25.9)

\color{blue}
\qbezier(-10,0.9)(-10,0.9)(190,174.1)
\qbezier(-10,199.1)(-10,199.1)(190,25.9)

\color{red}
\qbezier(190,0.9)(190,0.9)(-10,174.1)

\color{oliv}
\put(10,190){\huge $\infty$}

\normalcolor
\put(0,95){$B_1$}
\put(0,170){$B_2$}
\put(55,190){$B_3$}
\put(76,190){$B_4$}
\put(91,190){$B_5$}
\put(113,190){$B_6$}
\put(170,170){$B_7$}

\put(54,130){\small $1$}
\put(121,130){\small $2$}
\put(66,109){\small $3$}
%\put(78,109){\small $4$}
%\put(97,109){\small $5$}
%\put(109,109){\small $6$}
%\put(83,96){\small $7$}
%\put(93,96){\small $8$}
%\put(66,83){\small $9$}
%\put(76,83){\small $10$}
%\put(93.5,83){\small $11$}
%\put(107,83){\small $12$}
%\put(50,63){\small $13$}
%\put(119,63){\small $14$}
\put(109,109){\small $4$}
\put(83,96){\small $5$}
\put(93,96){\small $6$}
\put(66,83){\small $7$}
\put(107,83){\small $8$}
\put(50,63){\small $9$}
\put(119,63){\small $10$}

%right
\color{oliv}
\put(210,100){\line(1,0){200}}
%\multiput(285,0)(25,0){3}{\line(0,1){200}}
\put(285,0){\line(0,1){200}}
\put(310,0){\line(0,1){200}}
\put(335,0){\line(0,1){200}}

\color{blue}
\qbezier(410,128.86)(410,128.86)(210,13.4)
\qbezier(410,157.73)(410,157.73)(210,42.27)
\qbezier(410,186.6)(410,186.6)(210,71.14)

\qbezier(367.734,0)(310,100)(252.266,200)

\color{red}
\qbezier(210,128.86)(210,128.86)(410,13.4)
\qbezier(210,157.73)(210,157.73)(410,42.27)
\qbezier(210,186.6)(210,186.6)(410,71.14)

\qbezier(367.734,200)(310,100)(252.266,0)

\end{picture}
      \caption{$\A(12,1)$}
\label{fig:A(12,1)}
\end{figure}
\end{example}

%We can generalized Example \ref{ex:A(12,1)} and 
%Example \ref{ex:A(6,1)} as follows. 

%\begin{proposition}
%\label{prop:simplicial}
%Let $\A=\A(6m,1)$. Then 
%\begin{displaymath}
%\dim H^1(F_\A)_\lambda=
%\left\{
%\begin{array}{ll}
%1&\mbox{ for } \lambda=e^{\pm 2\pi i/3},\\
%0&\mbox{ for } \lambda^3\neq 1. 
%\end{array}
%\right.
%\end{displaymath}
%\end{proposition}

More generally, using Theorem \ref{thm:multinet} and 
Theorem \ref{thm:sharp}, 
we can prove that 
$\A(6m,1)$ is pure-tone. 
All other examples except for $\A(6m,1)$ in 
the catalogue \cite{gru-simp} (and \cite{cun}) satisfy 
$H^1(F_\A)_{\neq 1}=0$. It seems natural to pose the following. 

\begin{conjecture}
Assume that $c\A$ is a simplicial arrangement. Then the 
following are equivalent. 
\begin{itemize}
\item[(a)] $c\A=\A(6m,1)$ for some $m>0$. 
\item[(b)] $H^1(F_\A)_{\neq 1}\neq 0$. 
\item[(c)] $\A$ is pure-tone. 
\item[(d)] $c\A$ has a $k$-multinet structure for some $k\geq 3$. 
\item[(e)] $c\A$ has a $3$-multinet structure. 
\end{itemize}
\end{conjecture}

%\begin{remark}
%
%\end{remark}

\subsection{Zoo of non-trivial eigenspaces}
\label{sec:zoo}

\begin{example}
\label{ex:pappus}
Let $c\A$ be the Pappus arrangement (Figure \ref{fig:pappus}), 
so that $|c\A|=n+1=9$. Let $k=3$. Then $\RB_3(\A)=\{B_1, B_2, B_3\}$. 
By the expressions 
$$
\begin{array}{rlrcccccr}
\nabla(B_1)=&\sqrt{-3}(C_1&&+C_3&&&-C_9&&-C_{11})\\
\nabla(B_2)=&\sqrt{-3}(C_1&+C_2&+C_3&+C_4&-C_8&-C_9&-C_{10}&-C_{11})\\
\nabla(B_3)=&\sqrt{-3}(&C_2&&+C_4&-C_8&&-C_{10}&) 
\end{array}
$$
there is a unique relation 
$\nabla(B_1)-\nabla(B_2)+\nabla(B_3)=0$. Hence 
the Pappus arrangement is pure-tone. 
%We have 
%$\dim H^1(F_\A)_\lambda=1$ for $\lambda=
\begin{figure}[htbp]
\begin{picture}(400,150)(0,0)
\thicklines

\color{blue}
\put(280,120){\huge $\infty$}
%\qbezier(300,107.5)(300,107.5)(100,82.5)
\qbezier(300,107.5)(200,95)(100,82.5)
\put(100,60){\line(1,0){200}}

\normalcolor

\color{red}
\put(160,0){\line(0,1){150}}

\color{oliv}
\put(240,0){\line(0,1){150}}

\put(100,0){\line(1,1){150}}

\color{red}
\put(140,0){\line(1,1){150}}

\color{oliv}
\put(280,0){\line(-4,3){200}}

\color{red}
\put(320,0){\line(-4,3){200}}

\normalcolor
\put(105,140){$B_1$}
\put(195,140){$B_2$}
\put(260,140){$B_3$}

\put(165,97){\small $C_1$}
\put(220,105){\small $C_2$}
\put(170,83){\footnotesize $C_3$}
\put(212,86){\footnotesize $C_4$}
\put(161,73){\footnotesize $C_5$}
\put(195,73){\footnotesize $C_6$}
\put(228,73){\footnotesize $C_7$}
\put(173,63){\footnotesize $C_8$}
\put(216,63){\footnotesize $C_9$}
\put(165,44){\small $C_{10}$}
\put(223,44){\small $C_{11}$}

\end{picture}
      \caption{Pappus arrangement (Example \ref{ex:pappus})}
\label{fig:pappus}
\end{figure}
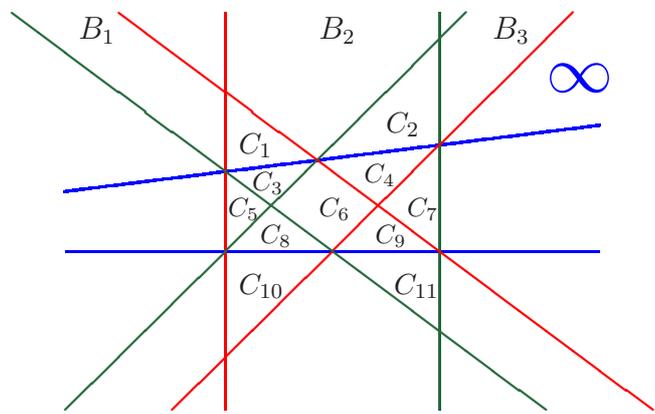

\end{example}

\begin{example}
\label{ex:page244-1}
(Taken from \cite[page 244]{gru-conf}.) 
%Figure 4.3.2 Left, 
Let $\A$ be as in the right-hand side of Figure \ref{fig:244-1}. 
Then $|c\A|=n+1=12$ and $\RB_3(\A)=\{B_1, \dots, B_5\}$. There 
is a unique linear relation 
$$
\nabla(B_1)
-\nabla(B_2)
+\nabla(B_3)
-\nabla(B_4)
=0
$$
($B_5$ does not appear). Hence $\A$ is pure-tone. 

\begin{figure}[htbp]
\begin{picture}(400,200)(0,0)
\thicklines
\color{blue}
\qbezier(90,0)(90,100)(90,200)

%2
\color{blue}
\qbezier(-10,107.143)(90,107.143)(190,107.143)

%3
\color{blue}
\qbezier(88.5714,,)(108.571,100)(128.571,200)

%4
\color{blue}
\qbezier(91.4286,0)(71.4286,100)(51.4286,200)

%5
\color{red}
\qbezier(-10.,42.265)(90.,100.)(190.,157.735)

%6
\color{oliv}
\qbezier(190.,42.265)(90.,100.)(-10.,157.735)

%7
\color{oliv}
\qbezier(40.5128,0)(98.2479,100)(155.983,200)

%8
\color{red}
\qbezier(139.487,0)(81.7521,100)(24.0171,200)

%9
\color{red}
\qbezier(3.79876,0)(117.587,100)(190.,163.639)

%10
\color{oliv}
\qbezier(-10.,163.639)(62.4134,100)(176.201,0)

%11
\color{oliv}
\qbezier(-10,153.053)(90.,119.225)(190.,85.3958)

%12
\color{red}
\qbezier(190,153.053)(90.,119.225)(-10.,85.3958)

\normalcolor
%%%% affine figure, right
%1
\color{blue}
\put(390,140){\huge $\infty$}

%2
\color{blue}
\qbezier(268.333,0)(310.,100)(351.667,200)

%3
\color{blue}
\qbezier(210.,133.846)(310.,115.385)(410.,96.9231)

%4
\color{blue}
\qbezier(210.,103.077)(310.,84.6154)(410.,66.1538)

%5
\color{red}
\qbezier(319.623,0)(319.623,100)(319.623,200)

%6
\color{oliv}
\qbezier(300.377,0)(300.377,100)(300.377,200)

%7
\color{oliv}
\qbezier(255.534,200)(338.868,100)(410.,14.641)

%8
\color{red}
\qbezier(210.,185.359)(281.132,100)(364.466,0)

%9
\color{red}
\qbezier(210.,181.067)(310.,110.357)(410.,39.647)

%10
\color{oliv}
\qbezier(210.,160.353)(310.,89.643)(410.,18.933)

%11
\color{oliv}
\qbezier(210.,15.856)(310.,105.028)(410.,194.199)

%12
\color{red}
\qbezier(210.,5.80081)(310.,94.9724)(410.,184.144)

\normalcolor
\put(210,12){\tiny $B_1$}
\put(305,10){$B_2$}
\put(390,5){$B_3$}
\put(408,23){$B_4$}
\put(400,75){$B_5$}

\end{picture}
      \caption{Example \ref{ex:page244-1}}
\label{fig:244-1}
\end{figure}
\end{example}

\begin{example}
\label{ex:page244-2}
(Taken from \cite[page 244]{gru-conf}.) 
%Figure 4.3.2 Left, 
Let $\A$ be as in the right-hand side of Figure \ref{fig:244-2}. 
Then $|c\A|=n+1=12$ and $\RB_3(\A)=\{B_1, \dots, B_4\}$. There 
is a unique linear relation 
$$
\nabla(B_1)
-\nabla(B_2)
+\nabla(B_3)
-\nabla(B_4)
=0. 
$$
Hence $\A$ is pure-tone. 
%Grunbaum page 244, Figure 4.3.2 Center 
%$n+1=12$. 
%\ref{fig:244-2}
\begin{figure}[htbp]
\begin{picture}(400,200)(0,0)
\thicklines
\color{oliv}
\qbezier(90,0)(90,100)(90,200)

%2
\color{red}
\qbezier(-10.,42.265)(90.,100)(190.,157.735)

%3
\color{blue}
\qbezier(190.,42.265)(90.,100)(-10.,157.735)

%4
\color{oliv}
\qbezier(-10.,125)(90.,125)(190.,125)

%5
\color{oliv}
\qbezier(-10.,106.25)(90.,106.25)(190.,106.25)

%6
\color{oliv}
\qbezier(-10.,68.75)(90.,68.75)(190.,68.75)

%7
\color{blue}
\qbezier(-3.81942,0)(53.9156,100)(111.651,200)

%8
\color{blue}
\qbezier(39.4819,0)(97.2169,100)(154.952,200)

%9
\color{blue}
\qbezier(61.1325,0)(118.868,100)(176.603,200)

%10
\color{red}
\qbezier(183.819,0)(126.084,100)(68.3494,200)

%11
\color{red}
\qbezier(140.518,0)(82.7831,100)(25.0481,200)

%12
\color{red}
\qbezier(118.868,0)(61.1325,100)(3.39746,200)

%% affine figure, right
%1
\color{oliv}
\put(380,185){\huge $\infty$}

%2
\color{red}
\qbezier(324.434,0)(324.434,100)(324.434,200)

%3
\color{blue}
\qbezier(295.566,0)(295.566,100)(295.566,200)

%4
\color{oliv}
\qbezier(210.,33.3333)(310.,100)(410.,166.667)

%5
\color{oliv}
\qbezier(272.5,0)(310.,100)(347.5,200)

%6
\color{oliv}
\qbezier(210.,153.333)(310.,100)(410.,46.6667)

%7
\color{blue}
\qbezier(210.,61.7863)(310.,88.453)(410.,115.12)

%8
\color{blue}
\qbezier(278.301,200)(345.801,110)(410.,24.4017)

%9
\color{blue}
\qbezier(210.,147.767)(310.,114.434)(410.,81.1004)

%10
\color{red}
\qbezier(210.,84.8803)(310.,111.547)(410.,138.214)

%11
\color{red}
\qbezier(210.,175.598)(274.199,90)(341.699,0)

%12
\color{red}
\qbezier(210.,118.9)(310.,85.5662)(410.,52.2329)

\normalcolor
\put(210,70){$B_1$}
\put(305,5){$B_2$}
\put(370,20){$B_3$}
\put(400,65){$B_4$}

\end{picture}
      \caption{Example \ref{ex:page244-2}}
\label{fig:244-2}
\end{figure}

\end{example}

\begin{example}
\label{ex:page44}
(Taken from \cite[page 44]{gru-conf}.) 
%Figure 4.3.2 Left, 
Let $\A$ be as in the right-hand side of Figure \ref{fig:44}. 
Then $|c\A|=n+1=15$ and $\RB_3(\A)=\{B_1, \dots, B_7\}$. There 
is a unique linear relation 
$$
\nabla(B_1)
-\nabla(B_3)
+\nabla(B_4)
-\nabla(B_6)
+\nabla(B_7)
=0
$$
($B_2$ and $B_5$ do not appear). Hence $\A$ is pure-tone. 

\begin{figure}[htbp]
\begin{picture}(400,200)(0,0)
\thicklines

\color{red}
\qbezier(90,0)(90,100)(90,200)

%%%2
\color{blue}
\qbezier(-10,42.265)(90,100)(190,157.735)

%%%3
\color{oliv}
\qbezier(-10,157.735)(90,100)(190,42.265)

%%%4
\color{blue}
\qbezier(88.4848,0)(97.5758,100)(106.667,200)

%%%5
\color{oliv}
\qbezier(91.5152,0)(82.4242,100)(73.3333,200)

%%%6
\color{red}
\qbezier(-10,20.24)(90,90.7677)(190,161.295)

%%%7
\color{oliv}
\qbezier(-10,62.0932)(90,108.311)(190,154.53)

%%%8
\color{red}
\qbezier(190,20.24)(90,90.7677)(-10,161.295)

%%%9
\color{blue}
\qbezier(190,62.0932)(90,108.311)(-10,154.53)

%%%10
\color{blue}
\qbezier(56.6065,0)(74.7252,100)(92.8438,200)

%%%11
\color{oliv}
\qbezier(123.394,0)(105.275,100)(87.1562,200)

%%%12
\color{oliv}
\qbezier(-10,48.1678)(90,84.0325)(190,119.897)

%%%13
\color{red}
\qbezier(-10,34.9828)(90,119.698)(184.789,200)

%%%14
\color{red}
\qbezier(190,34.9828)(90,119.698)(-4.78946,200)

%%%15
\color{blue}
\qbezier(190,48.1678)(90,84.0325)(-10,119.897)

%%%%% affine figure, right

%%%1
\color{red}
\put(390,160){\huge $\infty$}

%%%2
\color{blue}
\qbezier(317.217,0)(317.217,100)(317.217,200)

%%%3
\color{oliv}
\qbezier(302.783,0)(302.783,100)(302.783,200)

%%%4
\color{blue}
\qbezier(210,193.75)(310,143.75)(410,93.75)

%%%5
\color{oliv}
\qbezier(210.,106.25)(310.,56.25)(410.,6.25)

%%%6
\color{red}
\qbezier(355.769,0)(319.812,100)(283.855,200)

%%%7
\color{oliv}
\qbezier(284.54,0)(315.017,100)(345.495,200)

%%%8
\color{red}
\qbezier(336.145,0)(300.188,100)(264.231,200)

%%%9
\color{blue}
\qbezier(274.505,0)(304.983,100)(335.46,200)

%%%10
\color{blue}
\qbezier(210.,55.3418)(310.,78.0847)(410.,100.828)

%%%11
\color{oliv}
\qbezier(210.,99.1724)(310.,121.915)(410.,144.658)

%%%12
\color{oliv}
\qbezier(376.478,0)(315.206,100)(253.933,200)

%%%13
\color{red}
\qbezier(248.36,0)(319.554,100)(390.747,200)

%%%14
\color{red}
\qbezier(229.253,0)(300.446,100)(371.64,200)

%%%15
\color{blue}
\qbezier(366.067,0)(304.794,100)(243.522,200)

\normalcolor
%\put(220,140){$B_1$}
\put(215,75){$B_1$}
\put(236,0){\scriptsize $B_2$}
\put(275,0){\scriptsize $B_3$}
\put(305,0){\scriptsize $B_4$}
\put(343,0){\scriptsize $B_5$}
\put(366,0){\scriptsize $B_6$}
\put(400,50){$B_7$}

\end{picture}
     \caption{Example \ref{ex:page44}}\label{fig:44}
\end{figure}
\end{example}

\begin{definition}
\label{def:B3m}
For a positive integer $m\in\bZ_{>0}$, $\scB_{3m}$ is described 
as follows. Starting with a regular convex $2m$-gon in the 
Euclidean plane, $\scB_{3m}$ is obtained by taking $2m$ lines 
determined by the sides of the $2m$-gon together with 
$m$-diagonal lines connecting opposite vertices. (Note 
that $\scB_6$ is equivalent to the $A_3$-arrangement, 
see Figure \ref{fig:A3}.) 
\end{definition}

\begin{example}
\label{ex:B3m}
Using Theorem \ref{thm:multinet} and Theorem \ref{thm:sharp}, 
we can prove that the $\scB_{3m}$-arrangement is pure-tone. 

%\begin{example}
%\label{ex:page44}
%Grunbaum page 44, 
%$n+1=15$. This example is very big. 
\begin{figure}[htbp]
\begin{picture}(400,200)(0,0)
\thicklines

\color{red}
\qbezier(110,0)(110,100)(110,200)

%2
\color{blue}
\qbezier(70,0)(70,100)(70,200)

%3
\color{blue}
\qbezier(187.376,0)(114.721,100)(42.0671,200)

%4
\color{red}
\qbezier(137.933,0)(65.2786,100)(-7.37561,200)

%5
\color{red}
\qbezier(-10.,153.521)(90.,121.029)(190.,88.5373)

%6
\color{blue}
\qbezier(-10.,111.463)(90.,78.9708)(190.,46.4788)

%7
\color{blue}
\qbezier(190.,153.521)(90.,121.029)(-10.,88.5373)

%8
\color{red}
\qbezier(190.,111.463)(90.,78.9708)(-10.,46.4788)

%9
\color{red}
\qbezier(137.933,200)(65.2786,100)(-7.37561,0)

%10
\color{blue}
\qbezier(187.376,200)(114.721,100)(42.0671,0)

\color{oliv}
%11
\qbezier(90.,0)(90.,100)(90.,200)

%12
\qbezier(162.654,0)(90.,100)(17.3457,200)

%13
\qbezier(-10.,132.492)(90.,100)(190.,67.508)

%14
\qbezier(190.,132.492)(90.,100)(-10.,67.508)

%15
\qbezier(162.654,200)(90.,100)(17.3457,0)
\normalcolor

%%%%B_18

%1
\color{red}
\qbezier(330,0)(330,100)(330,200)

%2
\color{red}
\qbezier(290,0)(290,100)(290,200)

%3
\color{blue}
\qbezier(390.829,0)(333.094,100)(275.359,200)

%4
\color{blue}
\qbezier(344.641,0)(286.906,100)(229.171,200)

%5
\color{red}
\qbezier(210.,180.829)(310.,123.094)(410.,65.359)

%6
\color{red}
\qbezier(210.,134.641)(310.,76.906)(410.,19.171)

%7
\color{blue}
\qbezier(210.,120.)(310.,120)(410.,120)

%8
\color{blue}
\qbezier(210.,80.)(310.,80)(410.,80)

%9
\color{red}
\qbezier(410.,180.829)(310.,123.094)(210.,65.359)

%10
\color{red}
\qbezier(410.,134.641)(310.,76.906)(210.,19.171)

%11
\color{blue}
\qbezier(344.641,200)(286.906,100)(229.171,0)

%12
\color{blue}
\qbezier(390.829,200)(333.094,100)(275.359,0)

\color{oliv}
%13
\qbezier(336.795,0)(310.,100)(283.205,200)

%14
\qbezier(210.,200)(310.,100)(410.,0)

%15
\qbezier(210.,126.795)(310.,100)(410.,73.2051)

%16
\qbezier(410.,126.795)(310.,100)(210.,73.2051)

%17
\qbezier(210.,0)(310.,100)(410.,200)

%18
\qbezier(336.795,200)(310.,100)(283.205,0)

\end{picture}
     \caption{$\scB_{15}$ and $\scB_{18}$}\label{fig:B3m}
\end{figure}
\end{example}

%\cite{gru-conf} 
%p. 18, (subset of  p. 244), 
%p. 44 (Exceptional!) 

%Summary: $2$ infinite series. Not so much is known. 

%\subsection{Non-examples}

%Reye config, Metelka?, known small $4$-arrangements ($22$ lines). 

%Conjecture: Except for infinite series, they are 
%at most triple points? 

%\subsection{Pure tone conjecture}
%\label{sec:conj}

%observations: 
%\begin{itemize}
%\item order 3
%\item dim is at most one
%\item all are related to combinatorial pencil 
%\item all known examples are enough combin pencil plus 
%upperbound by narrow pair. 
%\item multiple points with multiplicity higher than 
%three is at most one. 
%\item except for two infinite series, all examples 
%shows at most triple points. 
%\end{itemize}

%problem: find infinite series. 
%Except for them, are they triple arr? 

%Alexander Polynomial. 

\medskip

\noindent
{\bf Acknowledgement.} 
Part of this work was done while the author 
was visiting Universidad de Zaragoza. 
The author gratefully acknowledges 
Professor E. Artal Bartolo and Professor J. I. Cogolludo-Agust\'in 
for their support, hospitality and encouragement. 
The author also thanks Michele Torielli for the comments to the 
preliminary version of this paper. 
This work is supported by a JSPS Grant-in-Aid for Young Scientists (B). 
%by . 
%especially to Prof. Mario Salvetti. 
%The author gratefully acknowledges the financial support 
%of JSPS. %Excellent Young Researchers Overseas Visit Program. 

%\noindent
%Masahiko Yoshinaga, \\
%Department of Mathematics, \\
%Hokkaido University, \\
%North 10, West 8, Kita-ku, \\
%Sapporo, 060-0810, \\
%JAPAN\\
%Email: yoshinaga@math.sci.hokudai.ac.jp

\end{document}